\documentclass[10pt]{amsart}

\usepackage{times}      
\usepackage{amssymb}    
\usepackage{amsmath}    
\usepackage{tikz-cd}
\usepackage{amsthm}
\usepackage{graphicx}
\usepackage{hyperref}
\usepackage{upgreek}
\usepackage{enumitem}

\usepackage{amssymb}
\usepackage{amsfonts} 
\usepackage{latexsym}   
\usepackage{amssymb}    
\usepackage{amsmath}    
\usepackage{amsbsy}
\usepackage{amsgen}
\usepackage{amsfonts}
\usepackage{array}

\usepackage{epsfig}
\usepackage{psfrag}
\usepackage[all]{xy}
\usepackage[margin=1.49in,footskip=0.25in]{geometry}

\makeatletter
\g@addto@macro\bfseries{\boldmath}
\makeatother

\makeatletter
\def\th@plain{%
	\thm@notefont{}
	\itshape 
}
\def\th@definition{%
	\thm@notefont{}
	\normalfont 
}
\makeatother

\usepackage{tkz-graph}
\SetUpVertex[Lpos=270,Math,LabelOut]	
\GraphInit[vstyle=Classic]				
\tikzset{VertexStyle/.append style = { minimum size = 4pt }}	
\tikzset{EdgeStyle/.style = {-,thick}}
\usetikzlibrary{arrows.meta}

\newtheorem{theorem}{Theorem}[section]
\newtheorem{sublemma}[theorem]{Sublemma}
\newtheorem{mainthm}[theorem]{Main Theorem}
\newtheorem{lemma}[theorem]{Lemma} 
\newtheorem{proposition}[theorem]{Proposition} 
\newtheorem{corollary}[theorem]{Corollary}
\newtheorem{remark}[theorem]{Remark}
\theoremstyle{definition}
\newtheorem{definition}[theorem]{Definition} 
\newtheorem*{ack}{Acknowledgments}
\newtheorem*{org}{Organization of the Paper}
\newtheorem*{mainproof}{Proof of Main Theorem \ref{main}}

\def\rrr{\mathbb{R}}
\def\ccc{\mathbb{C}}
\def\zzz{\mathbb{Z}}
\def\d{\partial}

\def\mstar{M^*}
\def\fstar{F^*}
\def\estar{E^*}

\def\ra{\rightarrow}

\DeclareMathOperator{\xt}{xt}
\DeclareMathOperator{\dist}{dist}
\DeclareMathOperator{\diam}{diam}
\DeclareMathOperator{\curv}{curv}
\DeclareMathOperator{\Fix}{Fix}

\let\originalleft\left
\let\originalright\right
\renewcommand{\left}{\mathopen{}\mathclose\bgroup\originalleft}
\renewcommand{\right}{\aftergroup\egroup\originalright}

\begin{document}

\title[Almost non-negatively curved $4$-manifolds with torus symmetry]{Almost non-negatively curved $4$-manifolds with torus symmetry}

\author[Harvey]{John Harvey}
\address{Department of Mathematics, Swansea University, Fabian Way, Swansea, SA1 8EN, UK.}
\email{j.m.harvey@swansea.ac.uk}

\author[Searle]{Catherine Searle}
\address{Department of Mathematics, Statistics and Physics, Wichita State University, Wichita, KS 67260, USA.}
\email{searle@math.wichita.edu}


\date{\today}

\begin{abstract} We prove that if a closed, smooth, simply-connected 4-manifold with a circle action admits an almost non-negatively curved sequence of invariant Riemannian metrics, then it also admits a non-negatively curved Riemannian metric invariant with respect to the same action. The same is shown for torus actions of higher rank, giving a classification of closed, smooth, simply-connected 4-manifolds of almost non-negative curvature under the assumption of torus symmetry.
\end{abstract}
\maketitle

\section{Introduction}

 The class of almost non-negatively curved manifolds contains precisely those manifolds which admit Riemannian metrics with a negative lower sectional curvature bound arbitrarily close to zero while maintaining an upper diameter bound. This is identical to the class of manifolds which collapse to a point with a lower sectional curvature bound. Understanding better the structure of almost non-negatively curved manifolds would be an important step toward a general theory of collapse with a lower curvature bound to any limit space. 
 Unfortunately, such manifolds are classified only  
in dimensions $3$ or lower, 
with the result in dimension $3$ due to Shioya and Yamaguchi \cite{SYa}. In dimension $2$, the result follows by the Gauss--Bonnet Theorem and is the same as the classification for non-negative curvature. 

Yamaguchi \cite{Ya} proved that for a manifold, $M$, of almost non-negative Ricci curvature, a finite cover of $M$ 
fibers over a $b_1(M)$-dimensional torus and, in the case where $b_1(M)=n$, $M$ is diffeomorphic to a torus. 

Thus, given the classification of almost non-negatively curved manifolds of dimension less than or equal to $3$, the only interesting case to understand in dimension $4$ is that of $b_1(M)=0$.  

The Grove Symmetry Program aims to classify manifolds with a lower curvature bound by assuming a certain amount of symmetry. Applying this principle to $4$-manifolds of almost non-negative curvature, we prove the following theorem.

\begin{mainthm}\label{main}
Let $S^1$  act
smoothly and effectively  on a closed, smooth, simply-connected $4$-manifold $M$. Let $\left\{
g_{n }\right\} _{n=1}^{\infty }$ be a sequence of Riemannian metrics on 
$M$ for which the $S^1$ action is isometric and suppose that $\{(M, g_{n})\}_{n=1}^{\infty}$ is almost non-negatively curved. Then $M$ admits a metric of non-negative curvature invariant under the same action.
\end{mainthm}

Bott has conjectured that all closed, simply-connected, non-negatively curved manifolds are rationally elliptic, and Grove has proposed that the conjecture might continue to hold in the case of almost non-negative curvature. This result provides evidence for extending the conjecture.

Recall that isometric $S^1$ actions on closed, simply-connected, non-negatively curved $4$-manifolds are classified by work of  Hsiang and Kleiner \cite{HK}, Kleiner \cite{K}, Searle and Yang \cite{SY}, Galaz-Garc\'ia \cite{GG1}, 
Galaz-Garc\'ia and Kerin \cite{GGK}, Grove and Searle \cite{GS}, and Grove and Wilking \cite{GW} as follows.

\begin{theorem}\cite{HK, K, SY, GG1, GGK, GS, GW}
Let $M$ be a closed, simply-connected, non-negatively curved 4-dimensional manifold with an isometric and effective $S^1$ action. Then $M$ is equivariantly diffeomorphic to $S^4$ or  $\ccc P^2$ with a linear $S^1$ action  or equivariantly diffeomorphic to one of $S^2\times S^2$ or $\ccc P^2\# \pm \ccc P^2$ with an $S^1$ sub-action of  a $T^2$ action induced by the standard $T^4$ action on $S^3\times S^3$.
\end{theorem}

The Main Theorem \ref{main} shows that this classification continues to hold in the case of almost non-negative curvature. The principal challenge in extending the result to almost non-negative curvature is in the proof of Lemma \ref{Xst}, where we rule out the possibility that five points are fixed. Any $S^1$ action on a closed, simply-connected, non-negatively curved 4-manifold which fixes five points can be shown to satisfy some very rigid geometric conditions, which yield a contradiction. The approach in almost non-negative curvature is necessarily very different.

\begin{mainproof}
	We recall that for smooth circle actions, the fixed-point set of the circle action, $\Fix(M; S^1)$, is of even codimension and that the Euler characteristic  $\chi(\Fix(M; S^1))=\chi(M^4)$ by work of Kobayashi \cite{Ko}. Since $M$ is closed and simply connected it follows that $\chi(M)>0$. Thus any circle action will have non-empty fixed-point set. The Main Theorem \ref{main}  is proven by considering two cases: Case 1,  where the action fixes only isolated fixed points, and Case 2, where the circle action is fixed-point homogeneous, that is, there is a codimension-two fixed-point set of the circle action. 
	We prove Case 1 in Proposition \ref{p:isolated} and Case 2 in Proposition \ref {FPHcase}. \qed
\end{mainproof}

We can generalize this result to classify all torus actions on almost non-negatively curved manifolds. The classification in the case of non-negative curvature continues to hold.

\begin{corollary}
Let $T^k$  act
smoothly and effectively  on a closed, smooth, simply-connected $4$-manifold $M$. Let $\left\{
g_{n }\right\} _{n=1}^{\infty }$ be a sequence of Riemannian metrics on 
$M$ for which the $T^k$ action is isometric and suppose that $\{(M, g_{n})\}_{n=1}^{\infty}$ is almost non-negatively curved. Then $k\leq 2$ and $M$ admits a metric of non-negative curvature invariant under the same action.

In particular, when $k=2$, $M$ is equivariantly diffeomorphic to one of $S^4$ or  $\ccc P^2$ with a linear $T^2$ action  or one of $S^2\times S^2$ or $\ccc P^2\# \pm \ccc P^2$ with a $T^2$ action induced by the standard $T^4$ action on $S^3\times S^3$, and when $k=1$, 
it is a sub-action of any of these.
\end{corollary}

\begin{proof}
By work of Parker \cite{Par}, there are no smooth actions of rank $k = 3$ on a closed, smooth, simply-connected $4$-manifold $M$, giving us the desired bound on the rank.

The Main Theorem \ref{main} provides the result for $k=1$, so that $M$ is diffeomorphic to  one of $S^4$, $\ccc P^2$, $S^2\times S^2$ and $\ccc P^2\# \pm \ccc P^2$. When $k=2$, clearly $M$ must be diffeomorphic to a manifold on this same list. However, it is known that every smooth $T^2$ action on these manifolds admits an invariant metric of non-negative curvature, by work of Orlik and Raymond \cite{OR1} for $S^4$ and $\ccc P^2$ and Galaz-Garc\'ia and Kerin \cite{GGK} for $S^2\times S^2$ and $\ccc P^2\# \pm \ccc P^2$, and the actions are classified as stated.
\end{proof}

\begin{org}
In Section \ref{s:prelim} we include notation and background  needed for the rest of the paper.   In Section \ref{s2'} we prove some results on triangles in almost non-negatively curved Alexandrov spaces and bound the number of boundary components of an almost non-negatively curved Alexandrov space. 
In Section \ref{s3} we prove Proposition \ref{p:isolated}  
and in Section \ref{s4} we prove Proposition \ref{FPHcase}: these two results combine to prove the Main Theorem \ref{main}.
\end{org}

\begin{ack} J.~Harvey thanks the staff and faculty of the Department of Mathematics at Wichita State University for their hospitality while a portion of this work was completed. Part of this research was carried out while he enjoyed the support of a Daphne Jackson Fellowship sponsored by the U.K.\ Engineering and Physical Sciences Research Council and Swansea University. C.~Searle  gratefully acknowledges support from grants from the U.S.\ National Science Foundation (\#DMS-1611780) and  from the Simons Foundation (\#355508, C.~Searle). The authors are grateful to the referee for their careful reading and suggested improvements.
\end{ack}

\section{Preliminaries}\label{s:prelim}

In this section we include   basic results and facts about transformation groups and Alexandrov spaces as well as notation that will be used throughout the paper.

\subsection{Transformation Groups}

Let $G$ be a compact Lie group acting by diffeomorphisms on a smooth manifold $M$. Recall that the {\em isotropy group} of a point $p\in M$ is the stabilizer of $p$ in $G$. We denote it by $G_p$ and note that it acts on $T_p M$. In this paper we will only consider the restricted action of $G_p$ on the unit sphere in $\nu_p M$, the space normal to the orbit through $p$.  
The action is called \emph{effective} if $\cap_{p \in M} G_p = \lbrace e \rbrace$.

In the case where $G=S^1$ there are three basic orbit types.  An orbit will be {\em principal}, {\em exceptional} or a {\em fixed point} if  its  isotropy subgroup is, respectively, trivial, finite of order $\geq 2$ or the full group $S^1$.

We will let $F$ denote the set of fixed points of the circle action in $M$ and $E$ the set of exceptional orbits. We let $\pi:M\rightarrow M/G=M^*$ be the orbit map and denote the images of $F$ and $E$ in $M^*$ by $F^*$ and $E^*$, respectively.

 Corollary IV.4.7 of Bredon \cite{Br} characterizes the orbit space, $M^*$, under the hypotheses of the Main Theorem \ref{main}.
\begin{lemma}\cite{Br}
\label{sconn} Let $G$ act on $M$ by cohomogeneity $3$, with $H_1(M; \zzz_2)=0$  and all orbits connected.  Then $M^*$ is a  $3$-manifold with or without boundary.
\end{lemma}
We can further characterize the orbit space using Proposition 3.1 of 
Fintushel \cite{F1},  which we recall here for the reader's convenience (see Figure \ref{f:fint}). 
\begin{proposition}\cite{F1}
\label{p:Fint} Let $S^1$ act smoothly on $M^4$, a closed, simply-connected $4$-manifold. 
Then the following hold.
\begin{enumerate}
\item In the case where $\d M^*\neq \emptyset$, 
$\d M^*\subset F^*$.
\item The set $F^*\setminus \d M^*$ is isolated.
\item The  set $E^*$ is a union of open arcs in $M^*$ and these arcs have closures with distinct endpoints in $F^*\setminus \d M^*$.
\end{enumerate}
\end{proposition}

\begin{figure}
	\begin{tikzpicture}
	\node[anchor=south west,inner sep=0] (image) at (0,0) {\includegraphics[width=0.6\textwidth]{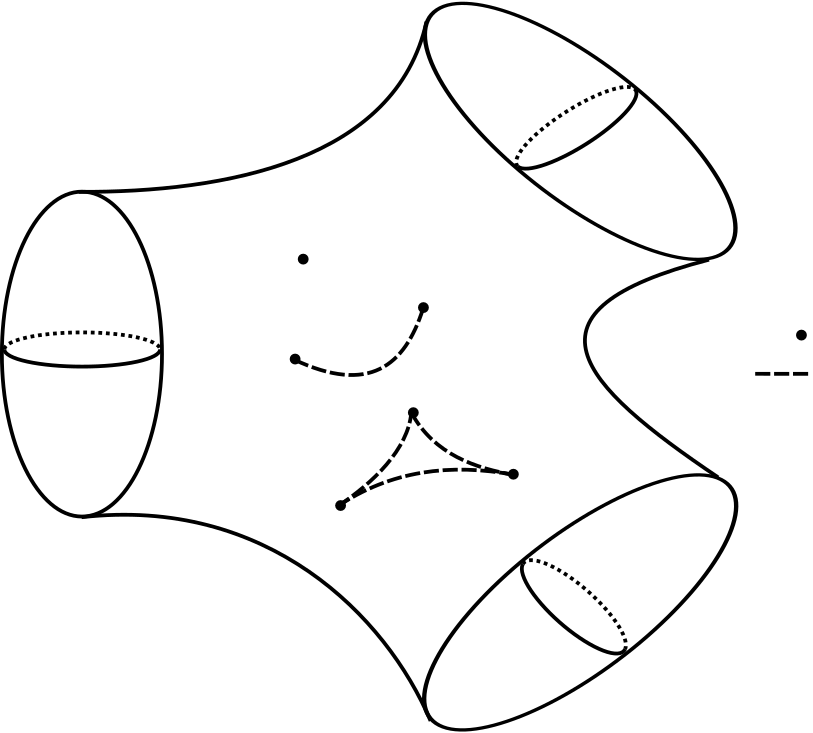}};
	\begin{scope}[
	x={(image.south east)},
	y={(image.north west)}
	]
	\node  at (0.9,0.9) {$\d \mstar \subset \fstar$};
	\node  at (-0.08,0.35) {$\d \mstar \subset \fstar$};
	\node  at (0.79,0.25) {$\d \mstar \subset \fstar$};
	\node  at (0.88,0.54) {$\fstar \setminus \d \mstar$};
	\node  at (0.88,0.49) {$\estar$};
	\end{scope}
	\end{tikzpicture}
	\caption{The orbit space $\mstar = M^4/S^1$ as described in Proposition \ref{p:Fint}.}
	\label{f:fint}
\end{figure}
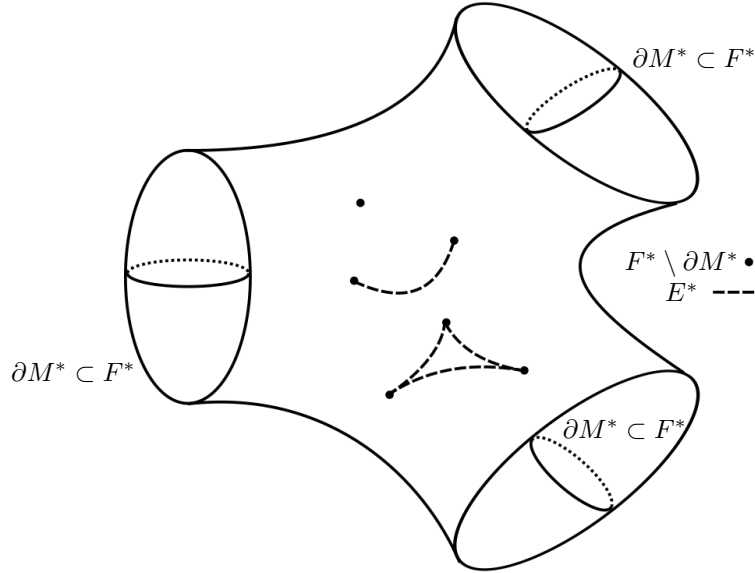

\subsection{Alexandrov Spaces}

Alexandrov geometry is a natural tool to use in studying isometric group actions in the context of bounded curvature because simple examples of Alexandrov spaces with $\curv \geq k$ include
\begin{itemize}
	\item Riemannian manifolds with $\sec \geq  k$ and
	\item quotients of Riemannian manifolds with $\sec \geq  k$ by closed groups of isometries.
\end{itemize}
The reader who is not familiar with Alexandrov spaces should find it sufficient for the purposes of this paper to consider only spaces of these two types.

A finite-dimensional \emph{Alexandrov space}  is a locally complete, locally compact  length space, with a lower curvature bound in the triangle comparison sense. In dimensions $n\geq 1$, the space is assumed to be connected, whereas in dimension $0$,  we allow a two-point space. Additionally, we will assume throughout this paper that the space is compact. 

There are a number of introductions to Alexandrov spaces to which the reader may refer for basic information (see, for example, Burago, Burago and Ivanov~\cite{BBI}, Burago, Gromov and Perelman~\cite{BGP}, Plaut~\cite{Pl}, and Shiohama~\cite{Sh}). 

A \emph{geodesic} in an Alexandrov space is a shortest path between two points. Unlike in Riemannian geometry,
geodesics are always globally length-minimizing.

The \emph{space of directions} of an $n$-dimensional Alexandrov space $X$ at a point $p$ is,
by definition, the completion of the 
space of geodesic directions at $p$ and is denoted by $\Sigma_p X$ or, where there is no confusion, $\Sigma_p$.  
The space of directions is a compact Alexandrov space of dimension $n-1$ with $\curv \geq 1$. In the case $X=M/G$, that is, where $X$ is the orbit space of an isometric action by a group $G$ on a Riemannian manifold $M$, $\Sigma_p$ is isometric to the quotient of the unit sphere in $\nu_q M$ by the isotropy subgroup $G_q$ for any $q \in \pi^{-1}(p)$.

Alexandrov spaces can have many types of singularities and among them are the \emph{extremal sets}, which are well reviewed in \cite{PetSC}. We shall not make this notion precise here; it is sufficient for the present work to mention that the closure of a stratum of an orbit space given by points representing orbits of the same type is an example of an extremal subset. Connected components of the boundary of an Alexandrov space are also examples of extremal subsets. At an isolated extremal point $p$ we have $\diam (\Sigma_p) \leq \pi/2$.

The class of Alexandrov spaces is closed under certain types of operations. Two of these will be useful in what follows. The first such operation is that of gluing along boundary faces. We summarize this result in the following theorem due to work of Perelman \cite{P1}, Petrunin \cite{PetQG}, and W\"orner \cite{Wo}. 
\begin{theorem}
	Let $X$ and $Y$ be two  Alexandrov spaces of the same dimension, both with $\curv \geq k$. Suppose that $A \subset X$ and $B \subset Y$ are connected components of the boundaries of $X$ and $Y$ respectively, or, more generally, are codimension-one extremal subsets. If $f \colon A \to B$ is an isometry with respect to the intrinsic metrics on $A$ and $B$, then $X \cup_f Y$ with the induced length  metric is also an Alexandrov space of $\curv \geq k$.
\end{theorem}

A second operation which preserves a lower curvature bound is that of taking the double branched cover over an extremal knot in $S^3$. The following lemma is a slight generalization of a result of Grove and Wilking \cite[Lemma 5.2]{GW}, which was originally stated only in the case $k=0$.  The proof by the present authors of a generalization of that result in a different direction \cite[Theorem A]{HS} shows how the curvature bound can be modified.

\begin{lemma}\cite{HS}	\label{ram} Let $X$ be an Alexandrov space of $\curv \geq k$ which is homeomorphic to $S^3$. Let $c$ be a simple closed curve in $X$ which is an extremal subset. Then the double branched cover of $X$ over $c$, $X_2(c)$, is also an Alexandrov space of $\curv \geq k$.
\end{lemma}

We now recall the definition of the $q$-extent of a metric space, which is useful in estimating the number of isolated fixed points of an isometric group action in the presence of a lower curvature bound. The $q$-extent is defined to be the maximal average distance between $q$ points, not necessarily distinct, in a metric space, and a {\em $q$-extender} is any set of $q$ points that achieves the $q$-extent. That is,  for any metric space $(X, d)$ and positive integer $q \geq 2$, we define the $q$-extent of $X$ to be
$$\xt_q(X) =  \binom q2^{-1}\sup_{x_1, \ldots, x_q \in X} \sum_{1\leq i,\, j \leq q} d(x_i,x_j).$$

Finally, we recall the following useful lemma from  \cite{SY}.
For two relatively prime integers $s,t$, denote by $X_{s, t}$  the orbit space of $S^3$ by the isometric circle action 
${\rm e}^{{\rm i}\theta} \cdot (z_2, z_2)= ({\rm e}^{{\rm i}s\theta}z_1, {\rm e}^{{\rm i}t\theta}z_2)$.

\begin{lemma}\label{xtlemma}\cite{SY}  The bounds
$$\xt_4(X_{s, t})\leq\pi/3 \textrm{ and } \xt_5(X_{s, t})\leq 3\pi/10$$
hold.
Moreover, given 4 distinct points in $X_{s, t}$ with $(|s|,|t|) \neq (1,1)$, 
$$\sum_{1\leq i<j\leq 4} \dist(x_i, x_j)<2\pi.$$
In particular, the 4-extent of 4 distinct points in $X_{s, t}$ is strictly less than $\pi/3$. 
\end{lemma}

\section{Alexandrov spaces of almost non-negative curvature}\label{s2'}

Since a Riemannian manifold with a lower curvature bound is also an Alexandrov space, we simply state the definition of almost non-negative curvature 
for Alexandrov spaces.

\begin{definition}\label{d:annc} We say that a
	sequence of Alexandrov spaces
	$\{(X, \dist_{n})\}_{n=1}^{\infty}$ 
	is \emph{almost non-negatively curved} if there is
	a fixed  $D>0$ so that 
	\begin{equation*}
		\diam\left( X,\dist_{n }\right) \leq D \text{ and }
		\curv\left( X,\dist_{n }\right) \geq -\frac{1}{n^2}.
	\end{equation*}
	
	We will also say that the topological space $X$ admits almost non-negative curvature (in the Alexandrov sense) or, less formally, that $X$ is an Alexandrov space of almost non-negative curvature.

	We can 
	always rescale the metrics, $\dist_{n}$, on $X$ so that each $(X, \dist_{n})$ 
	has diameter 1 and we will always do so. Let $(X_{\infty}, \dist_{\infty})$ denote the  limit space $\lim_{n \to \infty } (X, \dist_n)$. Then $\diam (X_{\infty},  \dist_{\infty})=1$.
\end{definition}

\begin{remark} When we talk about almost non-negatively curved \emph{manifolds}, we specify the manifold up to diffeomorphism. In Definition \ref{d:annc} we only specify the space up to homeomorphism and this can create ambiguities. For example, the round sphere $S^5$ and the double-suspension of the Poincar\'{e} homology sphere are markedly different as Alexandrov spaces but they are homeomorphic as topological spaces. 
It would thus be of interest to find a category intermediate between that of topological spaces and Alexandrov spaces which plays a similar role in the subject to that of the category of smooth manifolds.  
\end{remark}

In non-negative curvature, every geodesic triangle has angle sum at least $\pi$. As might be expected, in almost non-negative curvature this can be shown to be {\em almost} true.

\begin{lemma}\label{l:defect}
	Let $X$ be an Alexandrov space such that $\curv(X) \geq -k^2$ and $\diam (X) \leq 1$. Then the defect of any triangle in $X$ is bounded above by a function $\mu(k)$ with $\mu(k) = O(k^2)$.
\end{lemma}

\begin{proof}
	The defect of a triangle in a space $X$ with $\curv(X) \geq -k^2$ is bounded above by the defect of a triangle with the same side lengths in the hyperbolic plane of constant curvature $-k^2$,
	which is equal to the area of the triangle multiplied by $k^2$.
	It follows that, when $\diam (X) \leq 1$, an upper bound can be determined from the area of the largest triangle in the hyperbolic plane with side lengths at most $1$.  
	Bezdek \cite{Bezdek} showed that the area of a polygon in the hyperbolic plane with a given perimeter is maximized by a regular polygon.

	Computing the Taylor expansion of the hyperbolic law of cosines, we see that the angle in an equilateral triangle with side length $1$ is given by \[ \arccos \frac{\cosh^2 k - \cosh k}{\sinh^2 k} = \frac{\pi}{3} - O(k^2)<\pi/3, \] and the result follows.
\end{proof}

Perelman and Petrunin in \cite{PP2} showed that the triangle comparison condition can be framed in terms of the concavity properties of the distance function from any point $p$.

We say that a locally Lipschitz function $f:  \rrr \ra \rrr$ is \emph{$\lambda$-concave} if $\phi(t) = f(t) - \lambda t^2 / 2$ is a concave function. 
We can write $f'' \leq \lambda$, since this differential inequality holds in the barrier sense. 
As a consequence of the concavity of $\phi$, we have $\phi'_+ (t_0) \geq \phi'_- (t_1)$ for $t_0 < t_1$, where $\phi'_-$ and $\phi'_+$ are the left and right derivatives, respectively.
A function $f :  X \ra \rrr$ on a length space $X$ is \emph{$\lambda$-concave} if its restriction to every 
shortest path 
is $\lambda$-concave.

In an Alexandrov space with $\curv \geq -k^2$,
the function
$f_k=\rho_k \circ \dist(p,\cdot)$
with
$$ \rho_k(x) = \frac{1}{k^2}(\cosh(kx)-1) $$
is 
$(1+k^2f_k)$-concave. Note that in the hyperbolic plane of constant curvature $-k^2$, equality holds so that $f_k'' = 1+k^2f_k$.
Similar conditions hold for non-negative curvature bounds, but we omit them here.

Using this formulation we can now prove a lemma for a certain class of {\em thin triangles}; those with one short edge, such that the two endpoints of the edge are extremal points, as shown in Figure \ref{f:thin}. This lemma will be important to the proof of Proposition \ref{p:4bound}.

\begin{lemma}\label{thin} Let $X$ be an Alexandrov space with curvature bounded below by $-k^2$ and fix $d_{\min}>0$. Let $p_1, p_2, p_3\in X$ be three distinct points with 
$p_2, p_3$ extremal, $\dist(p_2, p_3) = \epsilon $ and 
$\dist(p_1, p_i)\geq d_{\min}$ for $i\in\{2, 3\}$. Then
$$\pi/2\geq \measuredangle p_1 p_2 p_3 \geq \pi/2- f(d_{\min},\epsilon, k),$$
where   
\begin{equation}\label{e:1}
 f(d_{\min}, \epsilon, k)
 = \epsilon \left(\frac{1}{d_{\min}} + O(k^2) \right) + O(\epsilon^3).
 \end{equation}
\end{lemma}

\begin{figure}
	\begin{tikzpicture}
	\node[anchor=south west,inner sep=0] (image) at (0,0) {\includegraphics[width=0.7\textwidth]{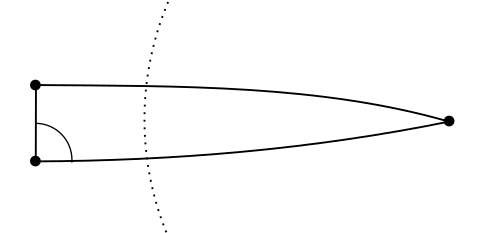}};
	\begin{scope}[
	x={(image.south east)},
	y={(image.north west)}
	]
	\node  at (0.96,0.485) {$p_1$};
	\draw [-{Latex[length=3mm,width=5mm]}, thick] (0.9,0.485) to (0.3,0.42);
	\node at (0.4,0.485) {$d_{\min}$};
	\node  at (0.07,0.70) {$p_3$};
	\node  at (0.07,0.24) {$p_2$};
	\node  at (0.02,0.47) {$\epsilon$};
	\node  at (0.11,0.38) {$\alpha$};
	\draw [-{Latex[length=3mm,width=3mm]}, thick] (0.02,0.51) to (0.02,0.635);
	\draw [-{Latex[length=3mm,width=3mm]}, thick] (0.02,0.43) to (0.02,0.3);
	\end{scope}
	\end{tikzpicture}
	\caption{A thin triangle. By Lemma \ref{thin}, if $p_2$ and $p_3$ are extremal then $\alpha \approx \pi/2$.}
	\label{f:thin}
\end{figure}
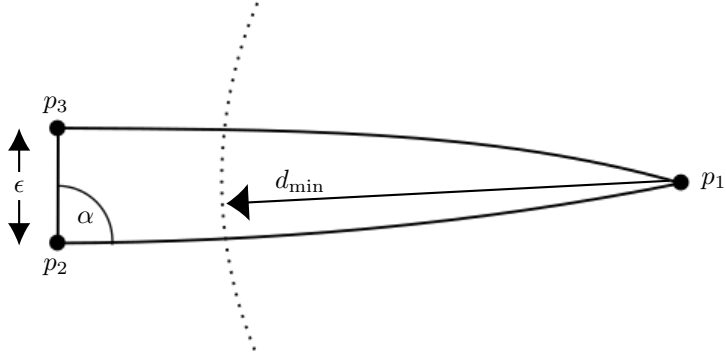

\begin{proof} Since $p_2$ is chosen to be extremal,   $\diam(\Sigma_{p_2}) \leq \pi/2$  and so $ \measuredangle p_1 p_2 p_3\leq \pi/2$.
	We proceed to demonstrate the lower bound.

	The function $f_k$ defined above satisfies $f_k'' \leq 1+k^2f_k$.
	Let $\gamma : [ 0, \epsilon] \to X$ be a geodesic from $p_2$ to $p_3$. Let $f = f_k \circ \gamma$ be the restriction of $f_k$ to the geodesic. Choose $R>0$ so that $f \leq R$ on $[ 0, \epsilon]$.
	It now follows that $f'' \leq 1+Rk^2$ on the geodesic, in other words $\phi(t)=f(t)-(1+Rk^2)\frac{t^2}{2}$ is concave.
	
	Let	$\alpha = \measuredangle p_1 p_2 p_3$. Then $$f'_+(0)= \frac{\sinh(k\dist(p_1, p_2))}{k} (-\cos \alpha)\leq - \frac{\sinh(kd_{\min})}{k} \cos \alpha.$$ Note that $\phi'_+(0) = f'_+(0)$.
	At the other end of $\gamma$, since $p_3$ is an extremal point, we have $\measuredangle p_1 p_3 p_2 \leq \frac{\pi}{2}$.  So $f'_-(\epsilon) \geq 0$ and hence $\phi'_-(\epsilon) \geq - (1+Rk^2) \epsilon$.
	Observe that the direction of the inequality is due to the fact that this is the angle at $p_3$ between a shortest path from $p_3$ to $p_1$ and a geodesic given by reversing $\gamma$.

	Now by concavity
	$\phi'_-(\epsilon) \leq \phi'_+(0)$
	so that
	$$-(1+Rk^2)\epsilon \leq -\frac{\sinh(kd_{\min})}{k} \cos \alpha,$$
	from which, using a Taylor series expansion for the last equality, we conclude that 
	$$  \cos \alpha \leq \frac{ k(1+Rk^2) }{\sinh(kd_{\min})}\epsilon = \epsilon \left(\frac{1}{d_{\min}} + O(k^2) \right).$$

	It follows that $$\measuredangle p_1 p_2 p_3 = \alpha \geq \frac{\pi}{2} - \epsilon \left(\frac{1}{d_{\min}} + O(k^2) \right) - O(\epsilon^3)$$ as required.
\end{proof}

We also need the following general proposition about $3$-dimensional Alexandrov spaces of almost non-negative curvature. The proof is very similar 
to that for non-negative curvature (see  \cite{K} and  \cite{SY}). 

\begin{proposition}\label{p:5bound}
	Let $\{(X, \dist_{n})\}_{n=1}^{\infty}$ be an almost non-negatively curved sequence of $3$-dimensional Alexandrov spaces. Then for sufficiently large $n$, $(X, \dist_n)$ can have at most five interior points with spaces of directions isometric to $S^3/S^1$.
\end{proposition}
\begin{proof} Let $S \subset X$ be the set of all such singular points and suppose that $\left|S\right|=6$. Write $S=\left\lbrace p_i \right\rbrace _{i=1} ^6$.

 Each of the 20 distinct triples which can be chosen from $S$ defines a triangle. Recalling that  $\curv(X,\dist_n) \geq -1/n^2$, it follows from Lemma \ref{l:defect} that each of these triangles has a total angle of at least $\pi - \mu(1/n)$.  So we may write \[\sum_{i, j, k} \measuredangle p_i p_j p_k \geq 20 (\pi - \mu(1/n)).\]
	
	On the other hand, 
	by considering the geometry of the space of directions $\Sigma_{p_j}$, we can obtain an upper bound for $\sum_{i,j,k} \measuredangle p_i p_j p_k$. Since by Lemma \ref{xtlemma} we have $\xt_5 (\Sigma_{p_j}) \leq 3 \pi / 10$ and there are ten angles based at $p_j$, we obtain \[\sum_{i, k} \measuredangle p_i p_j p_k \leq 3 \pi.\]
	
	Summing over all $j$, we obtain the inequality \[20 (\pi - \mu(1/n)) \leq 18 \pi,\] which, for large enough $n$, does not hold. The upper bound of five follows.
\end{proof}

It follows from the Soul Theorem for Alexandrov spaces \cite{P1} that a non-negatively curved Alexandrov space can have at most two boundary components. Here we show that the same result holds in almost non-negative curvature.

\begin{lemma}\label{boundary}
	An almost non-negatively curved Alexandrov space can have at most two boundary components.
\end{lemma}

\begin{proof}
As shown by Wong \cite{Wong}, the work of Liu and Shen \cite{liushen} bounding the Betti numbers of Alexandrov spaces  also bounds the number of boundary components in an Alexandrov space. That is, an Alexandrov space of dimension $n$, $\diam \leq D$ and $\curv \geq k$ can have at most $C(n, D, k)$ boundary components.

Since every almost non-negatively curved space admits a metric with $\diam\leq 1$ and $\curv \geq -1$, this implies a uniform bound $C(n)$ on the number of boundary components in an almost non-negatively curved Alexandrov space of dimension $n$.

However, if an almost non-negatively curved space had three boundary components, then by gluing copies of the space along boundary components it would be possible to produce almost non-negatively curved spaces with arbitrarily many boundary components, thus violating this bound. This contradiction demonstrates the result.
\end{proof}

\section{Case 1: The circle acts with isolated fixed points}\label{s3}

Recall that the strategy of proof for the Main Theorem \ref{main} is to consider two separate cases: Case 1,  where the circle action has only isolated fixed points, and Case 2, where there is a codimension-two fixed-point set. The goal of this section is to prove Proposition \ref{p:isolated}, which addresses Case 1.

The first step is to use Lemmas \ref{l:defect} and \ref{thin} to bound the number of isolated fixed points, which we will accomplish in Proposition \ref{p:4bound}. 
We observe that, for the case of non-negative curvature, geometric arguments were used in \cite{K} and \cite{GW} to rule out the presence of a fifth fixed point. Both papers rely on the fact that, at any of the five fixed points, precisely two of the six angles will equal $\pi/2$ and that the total angle at any vertex of the tetrahedron formed by any $4$ points must be $\pi$. However, once the rigidity of non-negative curvature is relaxed, of the six angles at any of the five fixed points, there could be four angles all of which are close to $\pi/2$, so that the argument breaks down in almost non-negative curvature.

In order to prove the upper bound of $4$ isolated fixed points, we need to prove a technical lemma, Lemma \ref{Xst}, which will follow once we have proven 
Sublemmas  \ref{Xst1}, and \ref{Xst2}. 

Before we attempt the proofs of Sublemmas \ref{Xst1}, and \ref{Xst2}, we need to better understand the geometry of $X_{s, t}$, $(|s|,|t|) \neq (1,1)$, which we recall is the quotient of $S^3$ under an isometric circle action  as in Lemma \ref{xtlemma}.
The diameter of $X_{s, t}$  is only achieved   
on the unique pair of antipodal points, $p$ and $q$, such that one or both are singular, depending on the values of $s$ and $t$.
Then, given any two points $v, w\in X_{s, t}$,  whose distance is sufficiently close to $\pi/2$,  we see that $v$  must be close to one of $p$ or $q$ and $w$ must be close to the other. That is, we obtain the following sublemma.

\begin{sublemma}\label{Xst0} Let $p, q\in X_{s, t}$ such that $\dist(p, q)=\pi/2$. Suppose that $p$ is a singular point in $X_{s, t}$. For any $\epsilon>0$ there is a $\delta>0$ such that if $v_0, v_1\in X_{s, t}$ with $\dist(v_0, v_1)\geq \pi/2 -\delta$, the following statements hold for some $i\in \{0, 1\}$, where we read $i+1$ modulo $2$.
\begin{enumerate}
\item  $|\dist(v_i, p)-\dist(v_{i+1}, q)|\leq \delta$;
\item $\dist(v_{i+1}, q)\leq (1+\epsilon)\delta$; and 
\item $\dist(v_{i}, p)\leq (2+\epsilon)\delta$.  \end{enumerate}
\end{sublemma}

 \begin{proof}
 Using the triangle inequality, we have 
 \begin{align*}
 \dist(v_i, v_{i+1}) &\leq \dist(v_i, p) +\dist(p, v_{i+1}) \\
 &=\dist(v_i, p) + \pi/2 -\dist(v_{i+1}, q),
 \end{align*}
 since $\dist(p, v_{i+1})) + \dist(v_{i+1}, q)=\pi/2$. 
 But $\dist(v_i, v_{i+1})\geq \pi/2 -\delta$ and therefore 
 $\dist(v_i, p)-\dist(v_{i+1}, q)\leq \delta$. This proves Part 1.
 
 To prove Part 2, we note that since  $\dist(v_i, v_{i+1})\geq \pi/2 -\delta$ and by compactness of $X_{s, t}$, and the uniqueness of $p$ and $q$, we have  $\dist(v_i, p), \dist(v_{i+1}, q)\leq C'(\delta)$, for some $i\in \{0, 1\}$, where  $\lim_{\delta\rightarrow 0}C'(\delta)=0$. It then follows from first variation of arc length that 
 $$ \dist(v_i, v_{i+1})\leq \dist(v_{i+1}, p) -\cos(\beta)\dist(v_i, p) + C(\delta)(\dist(v_i, p))^2,$$
 where $\beta=\measuredangle v_{i+1}pv_i$, and $C(\delta)=C(C'(\delta))$, with $\lim_{\delta\rightarrow 0}C(\delta)=0$. 
 Then, since $p$ is singular, it follows that $\beta\leq \pi/2$ and we have 
\begin{align*} \pi/2 -\delta\leq \dist(v_i, v_{i+1})&\leq \dist(v_{i+1}, p)  + C(\delta)(\dist(v_i, p))^2\\
&\leq \pi/2-\dist(v_{i+1}, q) +C(\delta)(\dist(v_i, p))^2\\
&\leq \pi/2-\dist(v_{i+1}, q) +C(\delta)(\dist(v_{i+1}, q) +\delta)^2,
\end{align*}
by Part 1.
Thus $$\dist(v_{i+1}, q)\leq \delta + C(\delta)(\dist(v_{i+1}, q) +\delta)^2.$$
Then either 
$$\dist(v_{i+1}, q)\leq \delta +4C(\delta)(\dist(v_{i+1}, q))^2 \textrm{ or } \dist(v_{i+1}, q)\leq \delta +4C(\delta)\delta^2.$$
 In the first case, since $1-4C(\delta)(\dist(v_{i+1}, q)>0$ and using a Taylor expansion, we see that 
 $$\dist(v_{i+1}, q)\leq \delta(1 +8C(\delta)\dist(v_{i+1}, q))\leq \delta(1+4\pi C(\delta)),$$ for small enough $\delta$. In the second case, noting trivially that $\delta<\pi$, we can show that the same inequality holds.    We set $4\pi C(\delta)<\epsilon$, and with this choice 
 we have proven Part 2.
 
 Part 3 follows by combining Parts 1 and 2.
 \end{proof}
 We now proceed to prove Sublemmas  \ref{Xst1}, and \ref{Xst2}.
Let $S=\left\lbrace p_i \right\rbrace_{i=1}^5$ be a set of five distinct points in an Alexandrov space $X$ of almost non-negative curvature
 with $\Sigma_{p_i}=X_{s_i, t_i}$ for each $p_i \in S$. 
We denote by $v_{ij} \in \Sigma_{p_i}$ the direction of a geodesic from $p_i$ to $p_j$.
	The set $S$ converges to some finite $S_{\infty} \subset X_{\infty}$ with $1 \leq \left| S_{\infty} \right| \leq 5$.

\begin{sublemma}\label{Xst1} Let $\{(X, \dist_{n})\}_{n=1}^{\infty}$ be an almost non-negatively curved sequence of 3-dimensional Alexandrov spaces.   Suppose that $S$ is defined as above and that $|S_{\infty}|\leq 4$ in $X_{\infty}$. Then there is a 
 $\delta > 0$ such that for some $k$ and for sufficiently large $n$,  $$\xt_4 (\lbrace v_{kl} : k \neq l \rbrace) \leq \pi/3 - \delta.$$ 
\end{sublemma}

\begin{proof}
The proof is broken into two cases: either $|S_{\infty}|\leq 2$ or $|S_{\infty}|\geq 3$.

	Let us define a convergence map $\pi : S \to S_{\infty}$ so that $\pi (p_i) = \lim_{n \to \infty} p_i$. By passing to a subsequence, we may assume that the preimages $\pi^{-1}(x)$, for each $x \in S_{\infty}$, satisfy $\dist_n (p_i, p_j) < 1/n$ for each $p_i, p_j \in \pi^{-1}(x)$. Choose $d$ so that $\dist_{\infty} (x,y) > 2 d$ for each pair of distinct points $x,y \in S_{\infty}$. 		
	We will denote by $\dist$ the distance function on any space of directions, $\Sigma_{p_j}$, and, for simplicity, we will omit the dependence of this space on $n$.
	Seeking a contradiction, suppose that for all $k$ we have 
	$$\xt_4 (\lbrace v_{kl} : k \neq l  \rbrace) \to \pi/3 \text{ as } n \to \infty.$$ 
	
	Consider first the case in which $\left| S_{\infty} \right| \leq 2$. Then there is some $x \in S_{\infty}$ which is the limit of at least three points.
	
	Suppose that exactly three points converge to $x$, so that $\pi^{-1}(x) = \lbrace p_1, p_2, p_3 \rbrace$. 
	Apply Lemma \ref{thin} to the three thin triangles given by $(p_j, p_5, p_4)$ for $j=1,2,3$, as shown in Figure \ref{f:twocluster}(a), with $\epsilon < 1/n$ and $d_{\min}=d$, as chosen above. 
	 Then, as shown in Figure \ref{f:twocluster}(b), and since $\measuredangle p_j, p_5, p_4 = \dist (v_{5j}, v_{54})$  in $\Sigma_{p_5}$, Equation \ref{e:1} of Lemma \ref{thin} gives us that $$\pi/2 - \dist (v_{5j}, v_{54}) \leq  f(d_{\min}, 1/n, 1/n)=  \frac{1}{n d} + O(1/n^3), \text{ for } 1 \leq j \leq 3.$$ 
	
	\begin{figure}
		\begin{tabular}{c@{\hspace{1cm}}c}
			\begin{tikzpicture}
			\node[anchor=south west,inner sep=0] (image) at (0,0) {\includegraphics[width=0.115\textwidth]{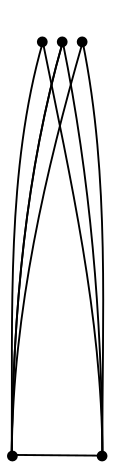}};
			\begin{scope}[
			x={(image.south east)},
			y={(image.north west)}
			]
			\node  at (0.38,0.95) {$p_1$};
			\node  at (0.58,0.95) {$p_2$};
			\node  at (0.78,0.95) {$p_3$};
			\node  at (0.23,0.06) {$p_4$};
			\node  at (0.79,0.06) {$p_5$};
			\end{scope}
			\end{tikzpicture}
			& 
			\begin{tikzpicture}
			\node[anchor=south west,inner sep=0] (image) at (0,0) {\includegraphics[width=0.2913\textwidth]{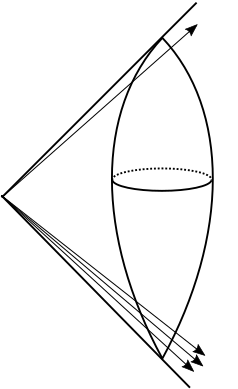}};
			\begin{scope}[
			x={(image.south east)},
			y={(image.north west)}
			]
			\node at (0.08,0.49) {$o$};
			\node  at (0.93,0.1) {$v_{51}$};
			\node  at (.91,0.06) {$v_{52}$};
			\node  at (0.88,0.02) {$v_{53}$};
			\node  at (0.89,0.94) {$v_{54}$};
			\node  at (0.7,0.45) {$\Sigma_{p_5}$};
			\end{scope}
			\end{tikzpicture}
			\\[0.5cm]
			(a) & (b) \\
		\end{tabular}
		\caption{One possible configuration with $|S_{\infty}|=2$. Where $|S_{\infty}| \leq 2$ there are always three thin triangles (a) guaranteeing the configuration shown in (b) within the cone on the
			space of directions.}
		\label{f:twocluster}
	\end{figure}
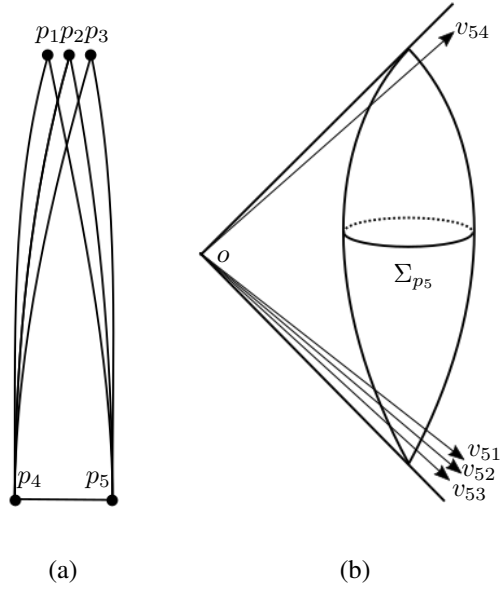
	
	In the case where $(|s_5|,|t_5|)=(1,1)$, then each $v_{5j}$, $1 \leq j \leq 3$, is close to the unique point antipodal to $v_{54}$. Hence the $v_{5j}$ are pairwise close to each other. In particular, it follows from the triangle inequality that 
	$$\dist (v_{5i}, v_{5j}) \leq  \frac{2}{n d} + O(1/n^3) \text{ for } 1 \leq i,j \leq 3.$$

	However, if $(|s_5|,|t_5|) \neq (1,1)$, it follows from  Parts 2 and 3 of Sublemma \ref{Xst0} that $v_{54}$ is close to some $\xi \in \Sigma_{p_5}$, which is either a singularity or the unique point antipodal to a singularity. In this case, each $v_{5j}$, $1 \leq j \leq 3$, is close to the unique point antipodal to $\xi$. 
	In particular, it follows from Parts 2 and 3 of Sublemma  \ref{Xst0} and the triangle inequality that 
	$$\dist (v_{5i}, v_{5j}) \leq  \frac{5}{n d} + O(1/n^3) \text{ for } 1 \leq i,j \leq 3.$$  
	
	In both cases, we obtain that
	$$\xt_4 (\lbrace v_{5l} : l \neq 5  \rbrace) \leq \frac{1}{2}\left(\frac{\pi}{2}+\frac{5}{nd}\right) + O(1/n^3) \xrightarrow[n \to \infty]{} \frac{\pi}{4} < \frac{\pi}{3},$$ a contradiction.

	Alternatively, $\lbrace p_1, p_2, p_3, p_4 \rbrace \subset \pi^{-1}(x)$. Let $y \in X_{\infty}$ so that $y\neq x$. Set $d'=\frac12 \dist_{\infty}(x,y)$.  
	Let $q_n \in (X, \dist_n)$ be such that $\lim_{n \to \infty}q_n = y$. We then apply Lemma \ref{thin} to the thin triangle given by $(q_n, p_1, p_j)$ 
	for $j=2,3,4$, now with $d_{\min}=d'$.
	We obtain at $\Sigma_{p_1}$ that, as before,  $$\dist (v_{1i}, v_{1j}) \leq \frac{5}{nd'}+O(1/n^3) \text{ for } 2 \leq i,j \leq 4$$ so that  $$\xt_4 (\lbrace v_{1l} : l \neq 1  \rbrace) \leq \frac{1}{2} \left(\frac{\pi}{2}+\frac{5}{nd'}\right) + O(1/n^3) \xrightarrow[n \to \infty]{} \frac{\pi}{4} < \frac{\pi}{3},$$ a contradiction.
	
	We now turn to the case  $3 \leq \left| S_{\infty} \right| \leq 4$.
Note first that, since these are extremal points, no three are \emph{collinear}, in the sense that no shortest path between two points of $S_{\infty}$ contains a third point of $S_{\infty}$.
	
	Suppose that $p_1, p_2 \in \pi^{-1}(x)$. Let $y=\pi(p_3)$ and $z=\pi(p_4)$. By renumbering, we may assume that $x$, $y$ and $z$ are all distinct. Then applying Lemma \ref{thin} to $(p_3, p_1, p_2)$ and $(p_4, p_1, p_2)$, we obtain that, at $\Sigma_{p_1}$, $$\pi/2 - \dist(v_{12},v_{1j}) \leq \frac{1}{n d} + O(1/n^3) \text{ for } j=3,4.$$ We can deduce as before that $$\dist(v_{13},v_{14}) \leq \frac{5}{n d} + O(1/n^3) \xrightarrow[n \to \infty]{} 0,$$ as shown in Figure \ref{f:threecluster}. It follows from the lower semi-continuity of angles in Alexandrov spaces (see Theorem 4.3.11 in \cite{BBI}) that $\measuredangle yxz = 0$, that is, the points $x$, $y$, and $z$ are collinear; a contradiction.
\end{proof}

	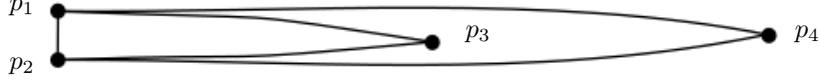
\begin{figure}
	\begin{tikzpicture}
	\node[anchor=south west,inner sep=0] (image) at (0,0) {\includegraphics[width=0.8\textwidth]{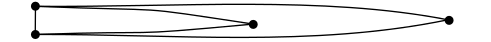}};
	\begin{scope}[
	x={(image.south east)},
	y={(image.north west)}
	]
	\node  at (0.57,0.54) {$p_3$};
	\node  at (0.96,0.54) {$p_4$};
	\node  at (0.03,0.9) {$p_1$};
	\node  at (0.03,0.1) {$p_2$};
	\end{scope}
	\end{tikzpicture}
	\caption{Where $|S_{\infty}| \geq 3$ the existence of two thin triangles guarantees collinearity.}
	\label{f:threecluster}
\end{figure}

\begin{sublemma}\label{Xst2} Let  $\{(X, \dist_{n})\}_{n=1}^{\infty}$, $S$ and $S_{\infty}$ be as in Sublemma \ref{Xst1}.   Suppose that $|S_{\infty}|=5$ and that there is a $p_i\in S$ such that $\Sigma_{p_i}$ is isometric to  $X_{s_i, t_i}$ with
 $(|s_i|,|t_i|) \neq (1,1)$.
Then there is a 
 $\delta > 0$ such that for some $k$ and for sufficiently large $n$,  $$\xt_4 (\lbrace v_{kl} : k \neq l \rbrace) \leq \pi/3 - \delta.$$ 
	\end{sublemma}

\begin{proof}	
	 Suppose, without loss of generality, that $i=5$, so that $\Sigma_{p_5}=X_{s, t}$ with $(|s|,|t|) \neq (1,1)$
	 and, seeking a contradiction, that $\xt_4 (\Sigma_{p_5}) = \pi/3$.
	By Lemma \ref{xtlemma}, the $4$-extender in $\Sigma_{p_5}$ is given by $$\lbrace w_1, w_2, w_3, w_4 \rbrace \subset \Sigma_{p_5} \text{ with } w_1=w_2, w_3=w_4, \dist(w_1,w_3) = \pi/2.$$
	
	It follows that, passing if necessary to a subsequence, we may assume that
	$$\dist(v_{51},v_{52}), \dist(v_{53},v_{54}) < 1/n \xrightarrow[n \to \infty]{} 0.$$ 
	
	It follows that the points in each of the two sets $\{\pi(p_5), \pi(p_1), \pi(p_2)\}$,  and  $\{\pi(p_5), \pi(p_3), \pi(p_4)\}$, are collinear. As noted in Sublemma \ref{Xst1}, this contradicts the extremality of $S_{\infty}$.
\end{proof}

Combining Sublemmas \ref{Xst1} and \ref{Xst2}, we obtain the following lemma.

\begin{lemma}\label{Xst} Let $\{(X, \dist_{n})\}_{n=1}^{\infty}$ be an almost non-negatively curved sequence of 3-dimensional Alexandrov spaces.   Suppose that $S = \left\lbrace p_i \right \rbrace_{i=1}^5$ is a set of five distinct points in $X$ 
 with $\Sigma_{p_i}=X_{s_i, t_i}$ for each $p_i \in S$. 
 Suppose there is a $j$ such that
 $(|s_j|,|t_j|) \neq (1,1)$.
Then there is a 
 $\delta > 0$ such that for some $k$ and for sufficiently large $n$,  $$\xt_4 (\lbrace v_{kl} : k \neq l \rbrace) \leq \pi/3 - \delta.$$ 
	\end{lemma}

	We are now in a position to prove that there are only $4$ isolated fixed points of the $S^1$ action.
\begin{proposition}\label{p:4bound}
	Let $S^1$  act
	smoothly and effectively  on a closed, smooth, simply-connected $4$-manifold $M$,  admitting an almost non-negatively curved sequence of $S^1$-invariant Riemannian metrics. If the fixed-point set of the action contains only isolated points, then there are at most four of these.
\end{proposition}

\begin{proof}
	By Proposition \ref{p:5bound} there are at most five isolated fixed points.  Suppose that there are exactly five fixed points, and consider their images, $\left\lbrace p_i \right\rbrace_{i=1}^5 = \fstar \subset \mstar$. These points define 10 triangles. Let $v_{jl} \in \Sigma_{p_j}$ be the direction of a geodesic from $p_j$ to $p_l$. If we apply similar arguments to those of Proposition \ref{p:5bound}, we will find only that the inequality \[10 (\pi - \mu(1/n)) \leq 10 \pi\] must hold, but since $\mu(1/n) > 0$ this is always true.
	
	The right hand side of this inequality stems from the statement $\xt_4 (\Sigma_{p_j}) \leq \pi / 3$.
	However, we can restrict our attention to calculating the potentially smaller value $\xt_4 (\lbrace v_{jl} : l \neq j \rbrace)$. 
	By Lemma \ref{Xst} we have that if, for some $i$, $\Sigma_{p_i}$ is isometric to some $X_{s, t}$ with $(|s|,|t|) \neq (1,1)$, the inequality
$$\xt_4 (\lbrace v_{ij} : j \neq i \rbrace) \leq \pi/3 - \delta$$
holds for some fixed $\delta>0$ and sufficiently large $n$. In that case, 
 \[10 (\pi - \mu(1/n)) \leq 10 \pi - \delta\] would hold, yielding a contradiction for sufficiently large $n$.
 
 It follows that, for all $i$, the space $\Sigma_{p_i}$ must be isometric to $X_{1,1}$. In other words, the action of the isotropy group $S^1$ on the unit sphere at each fixed point must be free.
	 	 
	The residue theorem of Bott \cite{Bott} then implies that the signature of $M$ is given by \[\sigma(M) = \frac{1}{3} \left( \sum_{i=1}^5 \pm 2 \right) \in \zzz,\] which can only be $\pm 2$. However, there is no closed, simply-connected 4-manifold with Euler characteristic 5 and signature $\pm 2$, and so there cannot be five fixed points, completing the proof.

\end{proof}

\begin{remark}\label{r:fifthpoint}
	Rather than using a signature argument, one can also rule out the possibility of a fifth fixed point by noting that the action is {\em semi-free}, that is, there are no points of finite isotropy. This is because Proposition \ref{p:Fint} requires the closure of any component of finite isotropy to intersect the set of fixed points, but as shown above the isotropy action on the unit sphere at each fixed point is free. By work of Church and Lamotke \cite{CL}, or from Fintushel's definition of a legally weighted 3-manifold \cite[5.2]{F1},  semi-free circle actions on closed, smooth, simply-connected 4-manifolds always have an even number of fixed points.
\end{remark}

\begin{proposition}\label{p:isolated}
	Let $S^1$  act
	smoothly and effectively  on a closed, smooth, simply-connected $4$-manifold $M$,  admitting an almost non-negatively curved sequence of $S^1$-invariant Riemannian metrics. If the action fixes only isolated fixed points, then there is an invariant metric of non-negative curvature.
\end{proposition}

\begin{proof}
	By Lemma \ref{sconn} and Proposition \ref{p:Fint}, 
	 the orbit space $\mstar$ is homeomorphic to $S^3$ and $\estar$ comprises arcs in $S^3$ joining the fixed points, of which, by Proposition \ref{p:4bound}, there are at most four. Since the number of fixed points gives $\chi(M)$, the Euler characteristic of $M$, the fact that $\chi(M) \geq 2$ for simply-connected 4-manifolds ensures that there are at least two fixed points.

	At this stage, the arguments made by Grove and Wilking in \cite{GW} in classifying isometric circle actions on non-negatively curved 4-manifolds  all carry through.
	We summarize these arguments for the sake of completeness. 
	
	First, for any closed curve $\gamma\subset\estar \cup \fstar$, we consider the double branched cover over $\gamma$, $M^*_2(\gamma)$. By Lemma \ref{ram}, $M^*_2(\gamma)$ is almost non-negatively curved. Moreover, its universal cover, $\widetilde{M^*_2}(\gamma)$, is also almost non-negatively curved. Observe that $\widetilde{M^*_2}(\gamma)$ must have at least $2|\pi_1(M^*_2(\gamma))|$ points with spaces of directions isometric to some $X_{s,t}$. By Proposition \ref{p:5bound}, $2|\pi_1(M^*_2(\gamma))|\leq 5$. Therefore $|\pi_1(M^*_2(\gamma))|\leq 2$. 	
	
	By Theorem C of \cite{GW}, we can use the topology of $M^*_2(\gamma)$ to recognize whether $\gamma$ is knotted. Namely, $\gamma$ is knotted if and only if the order of the fundamental group $|\pi_1(M^*_2(\gamma))| \geq 3$. It follows that $\gamma$ is the unknot.

This technique can also be used to show that $\gamma$ must pass through all of the $S^1$-fixed points, as follows. Recall that Proposition \ref{p:Fint} says that $\gamma$ cannot contain just one isolated fixed point and so, by counting singularities in $\widetilde{M^*_2}(\gamma)$, the only configuration we still need to rule out is where $\estar \cup \fstar=\{p\}\cup \alpha$, where $p$ is an isolated fixed point and $\alpha$ is a bi-angle consisting  of two distinct shortest paths in $\estar$ with common endpoints in $\fstar$. However, this configuration is ruled out by  Lemma 5.1 of Fintushel \cite{F1}, since the weight assigned to a closed curve is $0$ and to an isolated point $\pm 1$.

 Therefore $\mstar \cong S^3$ with $\fstar$ consisting of two, three or four isolated points and $\estar$ making up arcs between the points of $\fstar$,  so that any closed curve in $\estar \cup \fstar$ is unique, unknotted and contains all of $\fstar$  (cf. Theorem 2.5 in \cite{GW}). The decomposition of $M$ into two disk bundles described in Section $3$ of \cite{GW} does not depend on $M$ itself being non-negatively curved, but rather on the consequences of that fact for the topology of the orbit space, and so applies in our case.
\end{proof}

\section{Case 2: Fixed-point-homogeneous actions}\label{s4}

A fixed-point-homogeneous action is one where the orbit space has a boundary component corresponding to a component of the fixed-point set. By classifying these actions in Proposition \ref{FPHcase} we will complete the proof of the Main Theorem \ref{main}.

Since the action of $S^1$ on $M$ is fixed-point homogeneous, we see immediately that the orbit space $\mstar$  is a simply-connected, 
almost non-negatively curved $3$-manifold with boundary by Lemma \ref{sconn}. By Proposition \ref{p:Fint}, the union of $\d \mstar$ with a number of isolated singular points makes up $\fstar$,
while the points in $\estar$, the image of components of finite isotropy, comprise arcs joining the isolated points of $\fstar$. 

We can identify the topology of $\mstar$ using the following well-known fact.

\begin{lemma}
	If $Y$ is a simply-connected 3-manifold with $m$ boundary components, then $Y$ is homeomorphic to $S^3$ with $m$ copies of $D^3$ removed.
\end{lemma}

\begin{proof}
	By Lefschetz duality, $H_2(Y, \d Y) \cong H^1(Y) \cong 0$, the latter isomorphism holding since $Y$ is simply connected. The homology long exact sequence of the inclusion $\d Y \to Y$ then shows that $H_1(\d Y) \cong 0$, so that it is a union of copies of $S^2$.
	
	Now gluing in $m$ copies of $D^3$ along the $m$ boundary components produces, by the Van Kampen Theorem, a closed, simply-connected manifold, which by the resolution of the Poincar\'{e} Conjecture is homeomorphic to $S^3$.
\end{proof}

Recall that by Lemma \ref{boundary}, the orbit space $\mstar$ can have at most two boundary components, so we will characterize the orbit spaces on a case-by-case basis according to whether the boundary has one or two connected components.

\begin{lemma}\label{1}
	If $\mstar$ has one boundary component, then it is homeomorphic to $D^3$ and the action has at most two isolated fixed points. If $\estar \neq \emptyset$  then it is an arc between the two isolated points of $\fstar$.
\end{lemma}

\begin{proof}
	Since $\mstar$ is a copy of $S^3$ with one $D^3$ removed, it is homeomorphic to $D^3$.
	
	Consider a sequence of $S^1$-invariant metrics $\{g_n\}^\infty_{n=1}$ on $M$ with $\diam(M, g_n) = 1$ and $\curv(M, g_n)\geq -1/n^2$. Each induces an Alexandrov metric on $\mstar$ with $\diam\leq 1$ and $\curv \geq -1/n^2$, and the doubling of these Alexandrov spaces along the boundary produces another sequence of Alexandrov spaces, $(X, \dist_n)$, with $\diam\leq 2$ and $\curv\geq -1/n^2$. 
	
	Therefore $X$ also admits almost non-negative curvature. 
	By Proposition \ref{p:5bound}, for sufficiently large $n$, the metric on $X$ can have at most five points corresponding to isolated fixed points of the action on $M$, and therefore $M$ can contain at most two isolated fixed points.
	
	The set $\estar$ makes up a number of arcs between the isolated points of $\fstar$. Seeking a contradiction, suppose that there are two such arcs, creating a singular closed curve in $\mstar$. In the double, $X$, there are then two such curves, each with two points having small spaces of directions. Taking the double branched cover over one of these curves creates an almost non-negatively curved space with six points having small spaces of directions, violating Proposition \ref{p:4bound}.
\end{proof}

\begin{lemma}\label{2}
	If $\mstar$ has two  boundary components, then it is homeomorphic to $S^2 \times I$ and the action has no isolated fixed points or finite isotropy.
\end{lemma}

\begin{proof}
	As a copy of $S^3$ with two $D^3$s removed, $\mstar$ is homeomorphic to $S^2 \times I$.
	
	Two copies of $\mstar$ may be joined along a common boundary component to create a new space, also homeomorphic to $S^2 \times I$. Let us consider a space $X$ constructed by joining six copies of $\mstar$ in such a manner. Since the diameter is still finite, having increased by a factor of at most six, $X$  again  admits almost non-negative curvature in the Alexandrov sense.
	
	If $\mstar$ had an isolated fixed point, then for each $n$, $X$ would have six points with spaces of directions isometric to some $X_{s,t}$, violating Proposition \ref{p:5bound}. Since $\estar$ only appears as arcs joining isolated points of $\fstar$, there is also no finite isotropy.
\end{proof}

\begin{proposition}\label{FPHcase}
	Let $S^1$  act
	smoothly and effectively  on a closed, smooth, simply-connected $4$-manifold $M$,  admitting an almost non-negatively curved sequence of $S^1$-invariant Riemannian metrics.  If the action fixes a set of codimension two, then there is an invariant metric of non-negative curvature.
\end{proposition}

\begin{proof}
	By Lemmas \ref{1} and \ref{2}, the orbit space is either homeomorphic to (i) $D^3$ with up to two isolated fixed points and a possible arc joining them representing finite isotropy, or (ii) $S^2 \times I$ with no isolated fixed points or finite isotropy.
	
	These orbit spaces appear in the classification of simply-connected, non-negatively curved 4-manifolds with fixed-point-homogeneous circle actions in \cite{GG1}, and therefore they arise from actions equivariantly diffeomorphic to those described there.
\end{proof}



\end{document}